\documentclass[a4paper]{article}
\title{Lifting non-ordinary cohomology classes for $\SLthree$}
\date{}
\author{Chris Williams}

\usepackage{mathrsfs}
\usepackage{caption}

\usepackage{amsthm}
\usepackage{amsmath}
\usepackage{caption}
\usepackage{graphicx}
\usepackage{mathrsfs}
\usepackage{amsmath,amscd}
\usepackage{amssymb}
\usepackage{setspace}
\usepackage{cite}
\usepackage[balance,small,nohug,heads=vee]{diagrams}
\usepackage{bbm}
\usepackage[utf8]{inputenc}
\usepackage[T1]{fontenc}
\usepackage{lmodern}
\usepackage{bm}
\usepackage[backref=page]{hyperref}
\hypersetup{
    colorlinks,
    citecolor=black,
    filecolor=black,
    linkcolor=black,
    urlcolor=black
}

\usepackage{tabularx}
\usepackage{multirow}
\usepackage{array}

\usepackage[margin=1.5in]{geometry}

\renewcommand*{\backref}[1]{}
\renewcommand*{\backrefalt}[4]{%
    \ifcase #1 (Not cited.)%
    \or        (Cited on page~#2.)%
    \else      (Cited on pages~#2.)%
    \fi}

\newtheoremstyle{defplain}
  {1.0\baselineskip\@plus.2\baselineskip\@minus.2\baselineskip}
  {1.0\baselineskip\@plus.2\baselineskip\@minus.2\baselineskip}
  {}
  {}
  {\bfseries}
  {.}
  { }
  {}
\newtheoremstyle{slplain}
  {1.0\baselineskip\@plus.2\baselineskip\@minus.2\baselineskip}
  {1.0\baselineskip\@plus.2\baselineskip\@minus.2\baselineskip}
  {\itshape}
  {}
  {\bfseries}
  {.}
  { }
  {}
\newtheoremstyle{remarkplain}
  {1.0\baselineskip\@plus.2\baselineskip\@minus.2\baselineskip}
  {1.0\baselineskip\@plus.2\baselineskip\@minus.2\baselineskip}
  {}
  {}
  {\bfseries}
  {:}
  { }
  {}
\usepackage{fancyhdr}

\usepackage{titlesec}
\titlelabel{\thetitle.\quad}

\newenvironment{mfigure}[3]{
\begin{figure}[h!]
\centering
\captionsetup{justification=centering}
\caption{#1}
\makebox[\textwidth][c]{\includegraphics[width=#2\textwidth]{#3}}
\end{figure}}

\newcommand{\Q}{\mathbb{Q}}
\newcommand{\Qp}{\mathbb{Q}_p}

\newcommand{\Z}{\mathbb{Z}}
\newcommand{\Zp}{\mathbb{Z}_p}
\newcommand{\R}{\mathbb{R}}
\newcommand{\Proj}{\mathbb{P}^1}
\newcommand{\A}{\mathbb{A}}
\newcommand{\D}{\mathbb{D}}
\newcommand{\C}{\mathbb{C}}
\newcommand{\N}{\mathbb{N}}
\newcommand{\F}{\mathbb{F}}
\newcommand{\f}{\mathcal{F}}

\newcommand{\roi}{\mathcal{O}}

\newcommand{\pri}{\mathfrak{p}}
\newcommand{\roik}{\mathcal{O}_K}

\newcommand{\uk}{\underline{k}}

\newcommand{\AAA}{\mathcal{A}}

\newcommand{\pribar}{\overline{\pri}}

\newcommand{\newmod}[1]{\hspace{2pt}(\mathrm{mod}\hspace{2pt}#1)}

\newcommand{\Cp}{\mathbb{C}_p}

\newcommand*{\defeq}{\mathrel{\vcenter{\baselineskip0.5ex \lineskiplimit0pt
                     \hbox{\scriptsize.}\hbox{\scriptsize.}}}%
                     =}

\newcommand{\labelrightarrow}[1]{\mathrel{\mathop{\xrightarrow{\hspace*{1cm}}}^{#1}}}

\newcommand{\symb}{\mathrm{Symb}}
\newcommand{\Hom}{\mathrm{Hom}}
\newcommand{\SLt}{\mathrm{SL}_2}
\newcommand{\GLt}{\mathrm{GL}_2}

\newcommand{\hsp}{\hspace{12pt}}

\newcommand{\h}{\mathrm{H}}

\newcommand{\matr}{\begin{pmatrix}a&b\\c&d\end{pmatrix}}

\newcommand{\smallmatrd}[4]{\left(\begin{smallmatrix}#1 & #2\\#3 & #4\end{smallmatrix}\right)}
\newcommand{\binomc}[2]{\begin{pmatrix}#1\\#2\end{pmatrix}}

\newcommand\isorightarrow{\xrightarrow{
   \,\smash{\raisebox{-0.65ex}{\ensuremath{\scriptstyle\sim}}}\,}}

\theoremstyle{defplain}
\newtheorem{mdef}{Definition}[section]
\newtheorem*{mdef*}{Definition}
\theoremstyle{slplain}
\newtheorem{mthm}[mdef]{Theorem}
\newtheorem*{mthmnum}{Theorem}
\newtheorem{mlem}[mdef]{Lemma}

\newtheorem{mprop}[mdef]{Proposition}
\newtheorem{mcor}[mdef]{Corollary}

\newtheorem{mcla}[mdef]{Claim}

\theoremstyle{remarkplain}
\newtheorem*{mrem}{Remark}
\newtheorem*{mrems}{Remarks}
\newtheorem{mremnum}[mdef]{Remark}

\newtheorem*{mnot}{Notation}
\newtheorem{mnotnum}[mdef]{Notation}

\newcommand{\comp}[1]{\h^r(\Gamma,#1)}

\newcommand{\SLthree}{\mathrm{SL}_3}
\newcommand{\cts}{\mathrm{cts}}
\newcommand{\SLo}{\mathrm{SL}_1}
\newcommand{\threematrix}[6]{\begin{pmatrix}#1 & #2 & #3\\
0 & #4 & #5\\
0 & 0 & #6
\end{pmatrix}}
\newcommand{\disp}{\D_\lambda^P(\roi_L)}
\newcommand{\fil}{\f^N\disp}
\newcommand{\pitotal}{\threematrix{1}{0}{0}{p}{0}{p^2}}
\newcommand{\pione}{\threematrix{1}{0}{0}{p}{0}{p}}
\newcommand{\pitwo}{\threematrix{1}{0}{0}{1}{0}{p}}
\newcommand{\pr}{\mathrm{pr}}

\newcommand{\VAk}{V_k\otimes \A_k}
\newcommand{\AAk}{\A_k\widehat{\otimes}\A_k}
\newcommand{\VVkdual}{V_k^*\otimes V_k^*}
\newcommand{\VDk}{V_k^*\otimes\D_k}
\newcommand{\DDk}{\D_k\widehat{\otimes}\D_k}

\newcommand\blfootnote[1]{%
  \begingroup
  \renewcommand\thefootnote{}\footnote{#1}%
  \addtocounter{footnote}{-1}%
  \endgroup
}

\newcommand{\lb}{\\ \\}

\begin{document}

\maketitle
\blfootnote{\textup{2000} \textit{Mathematics Subject Classification}: \textup{11F75 (primary), 11F85 (secondary)}} 
\begin{abstract}
In this paper, we present a generalisation of a theorem of David and Rob Pollack. In \cite{PP09}, they give a very general argument for lifting ordinary eigenclasses (with respect to a suitable operator) in the group cohomology of certain arithmetic groups. With slightly tighter conditions, we prove the same result for non-ordinary classes. Pollack and Pollack apply their results to the case of $p$-ordinary classes in the group cohomology of congruence subgroups for $\SLthree$, constructing explicit overconvergent classes in this setting. As an application of our results, we give an extension of their results to the case of non-critical slope classes in the same setting.
\end{abstract}

\section*{Introduction}
\subsection*{Background}
Modular symbols are cohomological objects that are powerful computational and theoretical tools in the study of automorphic forms. Classical modular symbols are elements in the cohomology of a locally symmetric space with coefficients in some polynomial space, and in many cases, there are ways of viewing such elements in the group cohomology of certain arithmetic subgroups. For example, to a modular form of weight $k$ and level $\Gamma_0(N)$, one can attach an element of the group cohomology $\h^1(\Gamma_0(N),V_{k-2}(\C))$, where $V_{k-2}(\C)$ is the space of homogeneous polynomials in two variables over $\C$ of degree $k-2$. These cohomology groups are equipped with an action of the Hecke operators, and the association of a modular symbol to an automorphic form respects this action.
$\lb$
In \cite{Ste94}, Glenn Stevens developed the theory of \emph{overconvergent} modular symbols by replacing the space of polynomials with a much larger space, that of \emph{$p$-adic distributions}. There is a surjective Hecke-equivariant map from this space to the space of classical modular symbols (with $p$-adic coefficients). As a map from an infinite dimensional space to a finite dimensional space, this `specialisation map' must necessarily have infinite dimensional kernel, but in the same preprint, Stevens proved his \emph{control theorem}, which says that upon restriction to the `small slope eigenspaces', this specialisation map in fact becomes an isomorphism. This control theorem -- an analogue of Coleman's small slope classicality theorem -- has had important ramifications in number theory, being used to construct $p$-adic $L$-functions (see \cite{PS11} and \cite{PS12}) and Stark-Heegner points on elliptic curves (see \cite{Dar01} and \cite{DP06}). 
$\lb$
Such control theorems have now been proved in a variety of other cases, including -- but certainly not limited to -- for compactly supported cohomology classes attached to Hilbert modular forms by Daniel Barrera Salazar in \cite{Bar15}, for compactly supported cohomology classes attached to Bianchi modular forms in \cite{Wil17}, and for ordinary cohomology classes attached to automorphic forms for $\mathrm{SL}_3$ by David and Robert Pollack in \cite{PP09}. In the latter, Pollack and Pollack gave a very general argument for explicitly lifting group cohomology eigenclasses (of a suitable operator) in the \emph{ordinary} case, that is, when the corresponding eigenvalue is a $p$-adic unit. This general lifting theorem has been used in a variety of other settings, including in the work of Xevi Guitart and Marc Masdeu in the explicit computation of Darmon points (see \cite{GM14}).
$\lb$
Whilst control theorems do exist in wide generality -- for example, Eric Urban has proved a control theorem for quite general reductive groups in \cite{Urb11} -- these theorems are rarely constructive when we pass beyond the ordinary case. In this note, we generalise the (constructive) lifting theorem of Pollack and Pollack to non-ordinary classes. To do this, we use an idea of Matthew Greenberg in \cite{Gre07}, which the author found invaluable in developing the theory of overconvergent modular symbols over imaginary quadratic fields.
$\lb$
In the remainder of the paper, we give an application of this theorem. In particular, we give an extension of the results of Pollack and Pollack over $\SLthree$ to explicitly construct overconvergent eigenclasses in the non-critical slope case. There are subtleties in this situation that do not need to be considered in the ordinary case; in particular, whilst Pollack and Pollack lift with respect to the operator $U_p$ induced by the element
\[\pi \defeq \pitotal,\]
we instead consider the two elements
\[\pi_1 \defeq \pione, \hspace{12pt}\pi_2 \defeq\pitwo,\]
with $\pi_1\pi_2 = \pi$. These induce commuting operators $U_{p,1}$ and $U_{p,2}$ on the cohomology with $U_{p,1}U_{p,2} = U_p$. We then lift twice; once with respect to the operator $U_{p,1}$ to a module of `partially' overconvergent coefficients, then with respect to the operator $U_{p,2}$ to the module of fully overconvergent coefficients used by Pollack and Pollack. In each case, we get a notion of `non-critical slope', and by combining these two notions we get a larger range of `non-criticality' than if we had just considered the operator $U_p$. This is similar in spirit to the results of \cite{Wil17}, Section 6, where control theorems are proved for $\GLt$ over an imaginary quadratic field in which the prime $p$ splits as $\pri\pribar$. This is done by lifting first to a module of half-overconvergent coefficients with respect to $U_{\pri}$, then to a module of fully overconvergent coefficients with respect to $U_{\pribar}$. 
$\lb$
We give a very brief summary of the results in the case of $\SLthree$. First, we summarise the set-up:
\begin{mnotnum}\label{fullnot}
\begin{itemize}
\item[(i)] Let $\lambda = (k_1,k_2,0)$ be a dominant algebraic weight of the torus $T\subset \mathrm{GL}_3/\Q$, and let $\Gamma \subset \Gamma_0(p)$ be a congruence subgroup of $\SLthree$. 
\item[(ii)] Let $L/\Qp$ be a finite extension with ring of integers $\roi_L$. 
\item[(iii)]Let $V_\lambda(\roi_L)$ be the (finite-dimensional) space of classical coefficients over $\roi_L$, to be defined in Section \ref{classicalcoeffs}.
\item[(iv)] Let $V_\lambda^\star$ denote $V_\lambda$ with a twisted action, as defined in Definition \ref{twisted}.
\item[(v)] Let $\D_\lambda(\roi_L)$ be the (infinite-dimensional) space of overcovergent coefficients over $\roi_L$, to be defined in Section \ref{ovcgtcoeffs}.
\item[(vi)] Let $\rho_\lambda:\h^r(\Gamma,\D_\lambda(\roi_L)) \rightarrow \h^r(\Gamma,L_\lambda(\roi_L)$ be the specialisation map on the cohomology at $\lambda$, where $L_\lambda(\roi_L) \defeq \mathrm{Im}(\D_\lambda(\roi_L)) \subset V_\lambda^\star(L)$ is the image of specialisation on the coefficients, to be defined in Section \ref{specialisation}.
\end{itemize}
\end{mnotnum}
Then, in Theorem \ref{controltheorem}, we prove:
\begin{mthmnum}Suppose $\alpha_1,\alpha_2 \in \roi_L$ with $v_p(\alpha_1)<k_1-k_2+1$ and $v_p(\alpha_2)<k_2+1.$ Then the restriction
\[\rho_\lambda: \h^r(\Gamma,\D_\lambda(\roi_L))^{U_{p,i}=\alpha_i} \isorightarrow \h^r(\Gamma,L_\lambda(\roi_L)^{U_{p,i}=\alpha_i}\]
of the specialisation map to the simultaneous $\alpha_i$-eigenspaces of the $U_{p,i}$ operators is an isomorphism.
\end{mthmnum}

\begin{mfigure}{\emph{Graphic showing range of lifting for fixed $k_1 = k$ and varying $k_2$ (with dotted line $v_p(\alpha_1)+v_p(\alpha_2) = k+2$)}}{1.2}{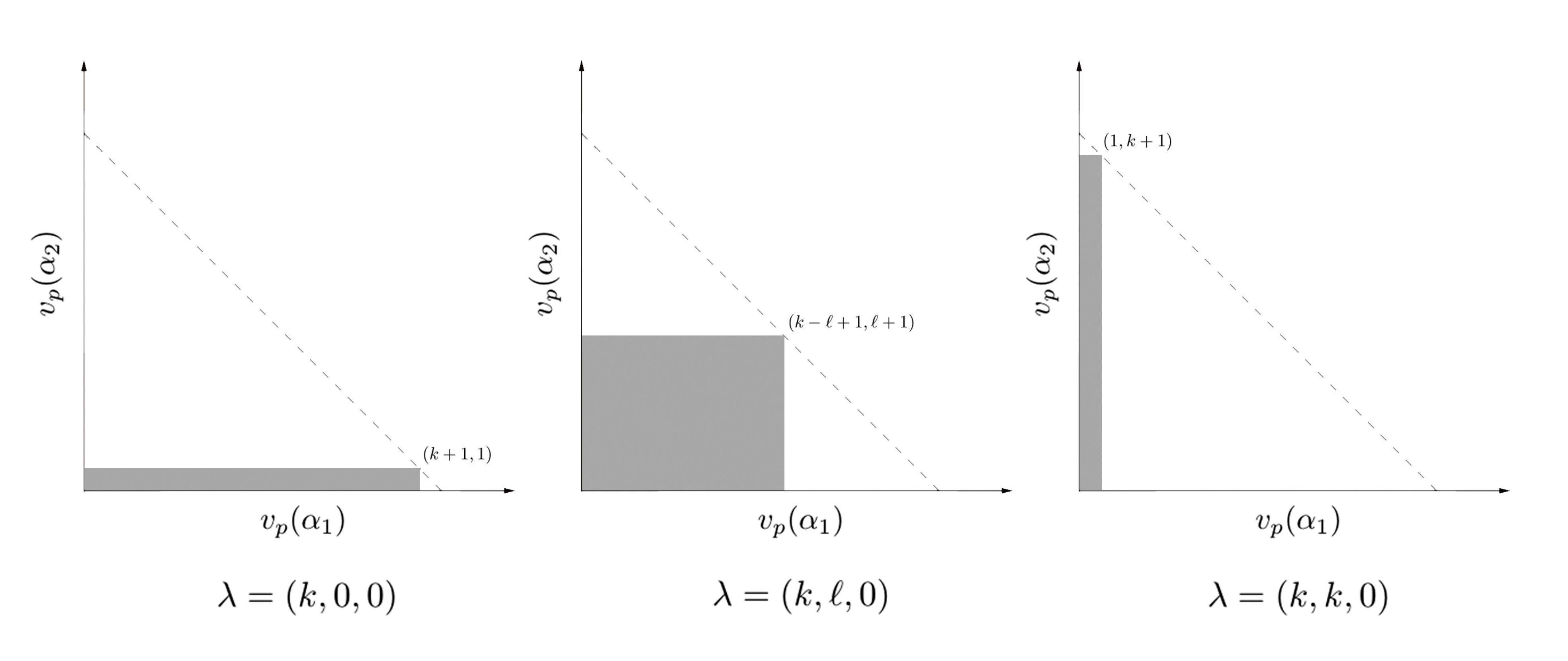}
\end{mfigure}

\subsection*{Summary of argument}
We give a brief summary of the argument we use to prove the general lifting theorem. The major component in the proof is showing that the specialisation map is surjective, in the process constructing an explicit lift of any element of the target space. Suppose we start with spaces $D$ and $V$, with actions of a group $\Gamma$ and an operator $U$, and suppose that $U$ also acts naturally on the group cohomology of these spaces. Suppose moreover that we have a surjection $\pr : D\rightarrow V$ that is equivariant with respect to the action of $\Gamma$ and $U$, inducing a map $\rho$ on the cohomology. We also assume that we can find a filtration $D \supset F^0D \supset F^1D \supset \cdots$ such that if we define $A^ND \defeq D/F^ND$, then we have $A^0D = V$. We also suppose that, among other conditions, we have $D \cong \lim_{\leftarrow}A^ND$. 
$\lb$
We then start with a $U$-eigenclass $\phi_0 \in \h^1(\Gamma,A^0D)$ with eigenvalue $\alpha$. Further assume that $\alpha$ is an algebraic integer (and hence can be thought of as living in the ring of integers of a finite extension of $\Qp$).
\begin{itemize}
\item[(i)]First suppose that $\phi_0$ is \emph{ordinary at $p$}, that is, suppose $\alpha$ is a $p$-adic unit. Then we take a cocycle $\varphi_0$ representing $\phi_0$, and lift it to \emph{any} cochain $\widetilde{\varphi_1}: \Gamma \rightarrow D$. As $\alpha $ is a unit, we can apply the operator $\alpha^{-1}U$ to this cochain. The magic is that $\varphi_1 \defeq \widetilde{\varphi_1}|\alpha^{-1}U \newmod{F^1D}$ is an $A^1D$-valued cocycle that is independent of choices, and thus defines a canonical lift of $\phi_0$ to a $U$-eigensymbol $\phi_1 \in \h^1(\Gamma,A^1D)$. Continuing in this vein, we get compatible classes $\phi_N \in \h^1(\Gamma,A^ND)$ for each $N$, and thus an eigenclass in the inverse limit $\Phi \in \h^1(\Gamma,D)$ that maps to $\phi_0$ under $\rho$.
\item[(ii)] For more general $\alpha$, we need a subtler argument. We would like to be able to apply the operator $\alpha^{-1}U$, but since $\alpha$ need not be a unit, we must strengthen our assumptions. In particular, we need the following:
\begin{itemize}
\item[(a)] A stronger condition on the filtration; namely, if $\mu\in F^ND,$ then $\mu|U\in \alpha F^{N+1}D.$
\item[(b)] An additional piece of data; namely, a $\Gamma$- and $U$-stable submodule $D^\alpha$ of $D$ such that if $\mu\in D^\alpha$, we have $\mu|U \in \alpha D$.
\end{itemize}
The benefit of this is that we can make sense of the operator $\alpha^{-1}U$ on cochains that have values in $D^\alpha$. We can run morally the same argument as above in this case. Unfortunately, the details of the argument become considerably more technical.
\end{itemize}
It is natural to ask when such conditions are satisfied. Condition (b) is relatively weak, and it seems reasonable to expect that a submodule $D^\alpha$ satisfying this condition exists in wider generality; in particular, when $D$ is a module of $p$-adic distributions on a finite number of variables, $D^\alpha$ can be defined by imposing a simple condition on the low degree moments. Condition (a), however, is stronger, and leads to the notion of \emph{small slope}. To illustrate this, consider the following examples of cases where such filtrations exist:
\begin{itemize}
\item One can find suitable filtrations in the cases of modular symbols attached to modular forms of weight $k+2$ over $\Q$ (see \cite{Gre07}). In this case, condition (a) is satisfied only if $v_p(\alpha)<k+1$, that is, if the modular form has \emph{small slope at $p$}. 
\item A similar result is given for modular forms over an imaginary quadratic field $K$ in \cite{Wil17}. In the case of weight $(k,k)$, and $p\roi_K = \pri\pribar$ split, the natural filtrations for $U_{\pri}$ and $U_{\pribar}$ satisfy condition (a) (with respect to $\alpha_{\pri}$ and $\alpha_{\pribar}$) only if $v_p(\alpha_{\pri}),v_p(\alpha_{\pribar}) < k+1.$ A more detailed description of these results is given in Section \ref{gl2}.
\end{itemize}
\subsection*{Structure}
In the first section, we describe the set-up of the theorem and the precise properties we require of our filtrations. In the second, we give a proof of our main theorem. In the third, we summarise the case of $\GLt$ over an imaginary quadratic field. In the fourth, we set up the case of $\SLthree$ by giving the relevant definitions of the various coefficient spaces and specialisation maps, and finally, in the fifth section, we define the filtrations we require in this case before stating the results for $\SLthree$.

\subsection*{Acknowledgements}
I would like to thank Marc Masdeu for encouraging me to publish the lifting theorem contained in this paper. I am also indebted to David Loeffler, as ever, for his very helpful suggestions as I worked on applying this theorem to the case of non-critical slope eigenclasses for $\SLthree$.
$\lb$
 The author was supported by an EPSRC DTG doctoral grant at the University of Warwick.

\section{Setup}
\begin{mnotnum}\label{setup}Suppose that we have:
\begin{itemize}
\item[(i)]A monoid $\Sigma$,
\item[(ii)] A group $\Gamma \leq \Sigma$, 
\item[(iii)]A ring $R$ and a right $R[\Sigma]$-module $D$,
\item[(iv)] An $R[\Sigma]$-stable filtration of $D$, given by $D \supset \f^0D \supset \f^1D \supset \cdots,$ such that if we define $\AAA^ND \defeq D/\f^ND$, then we have
\[\lim\limits_{\longleftarrow}\AAA^ND = D,\]
and where the $\f^ND$ have trivial intersection, and
\item[(v)] For some fixed $\alpha \in R$, a right $\Sigma$-stable submodule $D^\alpha$ of $D$, with $V^\alpha \defeq \mathrm{Im}(D^\alpha \rightarrow \AAA^0D).$
\end{itemize}
\end{mnotnum}
Note that for each $\gamma \in \Sigma$ such that $\Gamma$ and $\gamma^{-1}\Gamma\gamma$ are commensurable, and any $\Gamma$-module $\D$, we have an operator $U_\gamma$ on the cohomology group $\h^r(\Gamma,\D)$ defined in the usual way, that is, by the composition of the maps
\[\h^r(\Gamma,\D) \labelrightarrow{\text{res}} \h^r(\Gamma\cap\gamma^{-1}\Gamma\gamma,\D) \labelrightarrow{\gamma} \h^r(\Gamma\cap\gamma\Gamma\gamma^{-1},\D) \labelrightarrow{\text{cores}} \h^r(\Gamma,\D).\]


\begin{mthm}\label{liftingtheorem}
Suppose that $\alpha$ is a non-zero element of $R$, that $D^\alpha$ and $V^\alpha$and their corresponding cohomology groups have trivial $R$-torsion, and that for some $\pi \in \Sigma$, we have
\begin{itemize}
\item[(a)] If $\mu \in D^\alpha$, then $\mu|\pi \in \alpha D$, and
\item[(b)] If $\mu \in \f^ND,$ then $\mu|\pi \in \alpha \f^{N+1}D.$
\end{itemize}
Then the restriction of the natural map $\rho: \h^r(\Gamma,D^\alpha) \rightarrow \h^r(\Gamma,V^\alpha)$ to the $\alpha$-eigenspaces of the $U_\pi$ operator is an isomorphism.
\end{mthm}
\begin{mrem}This result is very similar to Theorem 3.1 of \cite{PP09}; their conditions are slightly weaker, but their conclusion requires $\alpha$ to be a unit. In their case, they do not require the condition on trivial $R$-torsion, and then prove that there is a unique eigenlift $\Phi$ of an eigensymbol $\phi$ that has $\mathrm{Ann}_R(\Phi) = \mathrm{Ann}_R(\phi)$. For simplicity, we have imposed this condition to ensure these annihilators are trivial. In the cases we consider, these conditions are satisfied.
\end{mrem}
We have natural $\Sigma$-equivariant projection maps 
\[\pr^N: D \longrightarrow \mathcal{A}^ND\]that induce $\Sigma$-equivariant maps 
\[\rho^N: \comp{D} \longrightarrow \comp{\mathcal{A}^ND},\]
(and hence $\rho\defeq \rho^0: \comp{D^\alpha} \rightarrow \comp{V^\alpha}$ by restriction) as well as maps $\pr^{M,N}:\mathcal{A}^MD \rightarrow \mathcal{A}^ND$ for $M \geq N$ that similarly induce maps $\rho^{M,N}.$ Thus we have an inverse system, and we have
\begin{align*}
\lim\limits_{\longleftarrow}\comp{\mathcal{A}^ND} = \comp{D}.
\end{align*}

First we pass to a filtration where the $\Sigma$-action is nicer. Define $\mathcal{F}^ND^\alpha = \mathcal{F}^ND \cap D^\alpha.$ This is a $\Sigma$-stable filtration of $D^\alpha$, since $D^\alpha$ is $\Sigma$-stable. It's immediate that if $\mu \in \mathcal{F}^ND^\alpha,$ then $\mu|\pi \in \alpha\mathcal{F}^{N+1}D^\alpha.$ Define $\mathcal{A}^ND^\alpha = D^\alpha/\mathcal{F}^ND^\alpha$, so that we have the following (where the vertical maps are injections):
\begin{diagram}
 D&&\rTo^{\pi^M}&&\mathcal{A}^MD&&\rTo^{\pi^{M,N}}&&\mathcal{A}^ND \\
\uTo&&&&\uTo&&&&\uTo \\
 D^\alpha &&\rTo^{\pi^M}&&\mathcal{A}^MD^\alpha &&\rTo^{\pi^{M,N}}&&\mathcal{A}^ND^\alpha.
\end{diagram}
Again, we see that
\begin{align}
\label{ila}\lim\limits_{\longleftarrow}\comp{\mathcal{A}^ND^\alpha} = \comp{D^\alpha}.
\end{align}

\begin{mnot}(The $U$ operator at the level of cochains). In \cite{PP09}, a description of the $U = U_\pi$ operator at the level of cochains is given. In particular, they take an explicit free resolution 
\[\cdots \labelrightarrow{\delta_3} F_2 \labelrightarrow{\delta_2} F_1 \labelrightarrow{\delta_1} F_0 \labelrightarrow{d_0} \Z \longrightarrow 0\]
of $\Z$ by right $\Z[\Gamma]$-modules; then, for a right $\Z[\Gamma]$-module $\D$, they use this to explicitly write down the spaces $C^r(\Gamma,\D) \defeq \mathrm{Hom}_\Gamma(F_r,\D)$ of cochains, $Z^r(\Gamma,\D) \defeq \mathrm{Ker}(d_{r}:C^r(\Gamma,\D) \rightarrow C^{r+1}(\Gamma,\D)$ of cocycles, and $B^r(\Gamma,\D) \defeq d_{r-1}(\Gamma,\D)$ of coboundaries, where $d_r$ is the obvious map induced by $\delta_r$. Then the group cohomology is defined as $\comp{\D} \defeq Z^r(\Gamma,\D)/B^r(\Gamma,\D)$.\\
\\
Now, $F_{*}^{\pi} \rightarrow \Z \rightarrow 0$ is a free resolution of $\Z[\pi^{-1}\Gamma \pi]$-modules, and as $\F_{*} \rightarrow \Z\rightarrow 0$ is also a free resolution of $\Z[\pi^{-1}\Gamma\pi]$-modules, there is a $\Z[\pi^{-1}\Gamma\pi]$-complex map $\tau_{*}$ from $F_{*}$ to $F_{*}^{\pi}$ lifting the identity map on $\Z$.\\
\\
Pick a set $\{\pi_i\}$ of coset representatives for $\Gamma\pi$ in $\Gamma\pi\Gamma$, noting that this is finite by commensurability. Then define $U:\mathrm{Hom}(F_r,D) \rightarrow \mathrm{Hom}(F_r,D)$ at the level of cochains by
\[(\varphi|U)(f_r) \defeq \sum_{i}\varphi(\tau_r(f_r\cdot\pi_i^{-1}))\cdot\pi\pi_i,\hsp \varphi \in \mathrm{Hom}(F_r,D), f_r\in F_r.\]
Pollack and Pollack prove (in Lemma 3.2) that this induces a map of chain complexes and hence a map of cohomology groups. In fact, this map is nothing other than $U_\pi$ as defined above.
\end{mnot}

\begin{mdef*}($U$-eigensymbols of eigenvalue $\alpha$). Since $\mathcal{A}^ND^\alpha$ may have non-trivial $\alpha$-torsion, we should make the statement ``$U_\pi$-eigensymbol in $\comp{\mathcal{A}^ND^\alpha}$'' more precise. By condition (a) of \ref{liftingtheorem}, if $\mu \in D^\alpha$, then $\mu|\pi \in \alpha D$. We can thus consider $\pi$ as a map from $D^\alpha$ to $D$ in a natural way, and define another map $V_\pi$ from $D^\alpha$ to $D$ by setting
\[x|V_\pi = y, \hsp \text{ where }x|\pi = \alpha y.\]
We see we have a formal equality of maps $\alpha V_\pi = \pi$ from $D^\alpha$ to $D$. Thus we get an operator
\[V\defeq V_\pi:\comp{D^\alpha} \longrightarrow \comp{D}\]
on the cohomology, so that we have an equality of operators $\alpha V = U$ as operators on $\comp{D^\alpha}$. There is also a canonical operator 
\[\varepsilon: \comp{D^\alpha} \longrightarrow \comp{D}\]
induced by the inclusion $D^\alpha \rightarrow D$. We see that if $\phi \in \comp{D^\alpha}$ satsifies $\phi|U = \alpha\phi$, then $\varepsilon(\phi) = \phi|V$ as elements of $\comp{D}$.

\begin{mrem}
The reason we don't simply just define $V = \alpha^{-1}U_\pi$ is that `dividing by $\alpha$' is not in general a well-defined notion on $D$. 
\end{mrem}
It is easy to see that for each $N$, $V$ gives rise to an operator $V_N : \mathcal{A}^ND^\alpha \rightarrow \mathcal{A}^ND$. Denote the canonical map $\comp{\mathcal{A}^ND^\alpha} \rightarrow \comp{\mathcal{A}^ND}$ by $\varepsilon_N$. We say an element $\varphi_N \in \comp{\mathcal{A}^ND^\alpha}$ is a \emph{$U$-eigensymbol of eigenvalue $\alpha$} if $\varepsilon_N(\varphi_N) = \varphi_N|V_N$ as elements of $\comp{\mathcal{A}^ND}$. Henceforth, when we talk about $U$-eigensymbols, it shall be assumed that the eigenvalue is $\alpha$.
\end{mdef*}

%
%
%
%

\section{Proof of Theorem \ref{liftingtheorem}}
\begin{proof}(Theorem \ref{liftingtheorem}). We first prove surjectivity. Take a $U$-eigensymbol $\phi_0$ of eigenvalue $\alpha$ in $\comp{V^\alpha} = \comp{\AAA^0D^\alpha}$. Suppose there exists a lift $\phi_N \in \comp{\AAA^{N+1}D^\alpha}$ of $\phi_0$ to a $U$-eigensymbol for some $N$. We prove that we can canonically lift $\phi_N$ to some $\phi_{N+1}$, and thus we will be done by induction and equation (\ref{ila}), as we have constructed an element in the inverse limit. We prove this in a series of claims.\\
\\
Take a cocycle $\varphi_N$ representing $\phi_N$, and lift it to a cochain $\varphi \in C_{\Gamma}(F_r,D^\alpha)$. We apply $V$ at the level of cochains, obtaining a cochain $\varphi|V:F_n\rightarrow D$. Define a cochain 
\[\tau_{N+1}:F_n \longrightarrow \AAA^{N+1}D\]
by composing this with the reduction map. This is in fact a cocycle; as $\varphi_N$ is a cocycle, $d\varphi$ takes values in $\f^ND^\alpha$, and thus as we have $d(\varphi|V) = (d\varphi)|V$ taking values in $\f^{N+1}D$ (by properties of $V$), it follows that $d\tau_{N+1} = 0$. Thus $\tau_{N+1}$ represents some cohomology class $[\tau_{N+1}]_D \in \comp{\AAA^{N+1}D}.$

\begin{mcla}\label{tauind}The cohomology class $[\tau_{N+1}]_D$ is independent of choices.
\end{mcla}
\begin{proof}
Suppose we take a different cochain $\widetilde{\varphi}$ lifting a different cocycle $\widetilde{\varphi_N}$ to a cochain taking values in $D^\alpha$. Then $[\rho^N(\varphi - \widetilde{\varphi})]_{D^\alpha} = 0$, where $\rho^N$ is the natural reduction map, as $\varphi_N$ and $\widetilde{\varphi_N}$ both represent $\phi_N$. Thus $[\varphi - \widetilde{\varphi}]_{D^\alpha} \in \comp{D^\alpha}$ is represented by a cocycle $\psi$ taking values in $\f^ND^\alpha$. Therefore $[\varphi - \widetilde{\varphi}]_{D^\alpha}|V$ is represented by $\psi|V$, which by examining the explicit action of $U$ on cochains we see to take values in $\f^{N+1}D$. After reduction (mod $\f^{N+1}D$), we see that 
\[[\tau_{N+1}]_D - [\rho^{N+1}(\widetilde{\varphi}|V)]_D = [\rho^{N+1}(\psi|V)]_D = 0,\]
which is the result.
\end{proof}

\begin{mcla}There exists a cocycle representing $[\tau_{N+1}]_D$ taking values in the smaller space $\AAA^{N+1}D^\alpha$.
\end{mcla}
\begin{proof}As $\phi_N$ is a $U$-eigensymbol, we know that as cocycles, $\varphi_N$ and $\tau_{N} \defeq \rho^N(\varphi|V)$ determine the same cohomology class in $\comp{\AAA^ND}$. Thus there exists some coboundary $b_N \in B^r(\Gamma,\AAA^ND)$ such that $\varphi_N = \tau_N + b_N$. Then by definition $b_N = d(c_N)$ for some $c_N \in C^{r-1}(\Gamma,\AAA^ND)$. Lift $c_N$ arbitrarily to a cochain $c_{N+1} \in C^{r-1}(\Gamma,\AAA^{N+1})$, and define $b_{N+1} \defeq d(c_{N+1}) \in B^r(\Gamma,\AAA^{N+1}D)$. Then 
\[\rho^{N+1,N}(\tau_{N+1} + b_{N+1}) = \tau_N + b_N = \varphi_N \in Z^r(\Gamma,\AAA^ND^\alpha).\]
Therefore it follows that $\varphi_{N+1} \defeq \tau_{N+1} + b_{N+1}$ takes values in the smaller space $\AAA^{N+1}D^\alpha$. As $\tau_{N+1}+b_{N+1} \in Z^r(\Gamma,\AAA^{N+1}D)$, it follows that $\varphi_{N+1} \in Z^r(\Gamma,D^\alpha)$. Thus $\varphi_{N+1}$ is the required cocycle to prove the claim.

\end{proof}

Define $\phi_{N+1} \defeq [\varphi_{N+1}]_{D^\alpha} \in \comp{\AAA^{N+1}D^\alpha}$ to be the $\AAA^{N+1}D^\alpha$-valued cohomology class determined by $\varphi_{N+1}$.

\begin{mcla}The cohomology class $\phi_{N+1}$ is independent of all choices.
\end{mcla}
\begin{proof}
Suppose we choose a different preimage $\widetilde{c_N}$ of $b_N$ under $d$, leading to a different $\widetilde{c_{N+1}}$ and $\widetilde{b_{N+1}}$, and thus a different $\widetilde{\varphi_{N+1}}$. Then 
\[\varphi_{N+1} - \widetilde{\varphi_{N+1}} = b_{N+1} - \widetilde{b_{N+1}} = d(c_{N+1} - \widetilde{c_{N+1}}).\]
As $\varphi_{N+1} -\widetilde{\varphi_{N+1}}$ takes values in $\AAA^{N+1}D^\alpha$, so must $c_{N+1} - \widetilde{c_{N+1}}$; hence $b_{N+1} - \widetilde{b_{N+1}} \in B^r(\Gamma,\AAA^{N+1}D^\alpha)$, so that 
\[[\varphi_{N+1}]_{D^\alpha} = [\widetilde{\varphi_{N+1}}]_{D^\alpha} \in \comp{\AAA^{N+1}D^\alpha}.\]
Thus they also determine the same cohomology class, namely $[\tau_{N+1}]_D$, in $\comp{\AAA^{N+1}D}.$ As the cohomology class $[\tau_{N+1}]_D$ is also uniquely determined by Claim \ref{tauind}, we're done. 
\end{proof}

\begin{mcla}$\phi_{N+1}$ is a $U$-eigensymbol with eigenvalue $\alpha$.
\end{mcla}
\begin{proof}It's clear that the representative $\varphi_{N+1}$ of $\phi_{N+1}$ is a lift of $\varphi_N$, by definition. Thus any lift $\varphi$ of $\varphi_{N+1}$ to a cochain taking values in $D^\alpha$ is also a lift of $\varphi_{N}$, and accordingly, it follows that 
\[\phi_{N+1}|V_{N+1} \defeq [\rho^{N+1}(\varphi|V)]_D = [\tau_{N+1}]_D.\]
Also by definition, $\varphi_{N+1}$ and $\tau_{N+1}$ represent the same elements of $\comp{\AAA^{N+1}D}$, so that $\varepsilon(\phi_{N+1}) = [\tau_{N+1}]_D$. Combining the two equalities gives $\varepsilon(\phi_{N+1}) = \phi_{N+1}|V_{N+1}$, which is the required result.
\end{proof}

Thus we obtain surjectivity. Take some $U$-eigensymbol $\phi_0 \in \comp{V^\alpha} = \comp{\AAA^0D^\alpha}$, and for each $N \in \N$, lift it to a $U$-eigensymbol $\phi_N$ using the above method. Then we obtain an element of the inverse limit $\lim_\leftarrow\comp{\AAA^ND^\alpha},$ which we know is isomorphic in a natural way to $\comp{D^\alpha}$. This element is thus a $U$-eigensymbol that maps to $\phi_0$ under the specialisation map.\\
\\
It remains to prove injectivity. Suppose $\phi \in \ker(\rho)$; we aim to show that $\phi = 0$. Consider the exact sequence
\[0 \longrightarrow \f^0D^\alpha \longrightarrow D^\alpha \longrightarrow V^\alpha \longrightarrow 0.\]
This leads to a long exact sequence of cohomology
\[\cdots \comp{\f^0D^\alpha} \longrightarrow \comp{D^\alpha} \labelrightarrow{\rho} \comp{V^\alpha} \longrightarrow \cdots,\]
and accordingly any element of $\ker(\rho)$ must lie in the image of $\comp{\f^0D^\alpha}$. This is the same as saying $\phi$ can be represented by a cocycle $\varphi$ taking values in $\f^0D^\alpha$. We now conclude using:
\begin{mcla}Let $\phi \in \comp{D^\alpha}$ be represented by a cocycle $\varphi$ taking values in $\f^0D^\alpha$. If $\phi$ is a $U$-eigensymbol, then $\phi = 0$.
\end{mcla}
\begin{proof}
We consider $\varepsilon(\phi) = [\varphi]_D$, which is also a $U$-eigensymbol. It thus makes sense to apply the operator $V$ to $[\varphi]_D$, for which it is a fixed point. By condition (b) of Theorem \ref{liftingtheorem}, the $V$ operator takes $\f^ND$ to $\f^{N+1}D$; therefore, as $[\varphi]_D$ is represented by $\varphi|V^N$ for any $N$ (by the eigensymbol property), we see that for each $N$, the symbol $[\varphi]_D$ is represented by a cocycle taking values in $\f^ND$. But the intersection of the $\f^ND$ is trivial by assumption. Thus $\varepsilon(\phi) = [\varphi]_D$ is 0.\\
\\
It remains to prove that the map $\varepsilon$ is injective. We now know that $\varphi$ is a coboundary in $C^n(\Gamma,D)$, so that there exists some $c \in C^{n-1}(\Gamma,D)$ with $\varphi = d(c)$. But as $\varphi$ takes values in $D^\alpha$, it follows that $c$ must also take values in $D^\alpha$. Thus $\varphi$ is also a coboundary in $C^n(\Gamma,D^\alpha)$, and $\phi = [\varphi]_{D^\alpha} = 0$, as required.
\end{proof}

This completes the proof of Theorem \ref{liftingtheorem}.
\end{proof}

%
%
%
%
\section{Application to $\GLt\times\GLt$}\label{gl2}
As an example of where this theorem applies, we give a brief summary of the case of $\GLt\times\GLt$, which is conceptually easier to understand than the case of $\SLthree$. In particular, we present the results in a concrete setting in the style of \cite{Wil17}, where these results were first proved. Recall the set-up:
\begin{mnot} Let $K$ be an imaginary quadratic field with ring of integers $\roi_K$, and let $p$ be a rational prime that splits as $\pri\pribar$ in $K$. Let $\Gamma \subset \Gamma_0(p)\subset \SLt(\roi_K)$ be a congruence subgroup. Let $\Sigma$ denote the set of complex embeddings of $K$, and let $\lambda = (k,k) \in Z[\Sigma]$ be a weight, where $k$ is non-negative. Let $L/\Qp$ be a finite extension with ring of integers $\roi_L$.
\end{mnot}
\subsection{Coefficient Modules}
\begin{mdef}Let $V_k(\roi_L)\defeq\mathrm{Sym}^k(\roi_L^2)$ be the space of homogeneous polynomials in two variables of degree $k$ over $\roi_L$.
\end{mdef}
We can identify $V_k(\roi_L)\otimes V_k(\roi_L)$ with a space of polynomial functions on $\roi_K\otimes_{\Z}\Zp$ in a natural way. 
\begin{mdef}Let $\A_k(\roi_L)\defeq \roi_L\langle z\rangle$ be the Tate algebra over $\roi_L$, that is, the space of power series in one variable whose coefficients tend to zero as the degree tends to infinity.
\end{mdef}
\begin{mrem}For ease of notation, we will henceforth drop $\roi_L$ from the notation. All tensor products are over $\roi_L$.
\end{mrem}
Let $\Sigma_0(p) \subset M_2(\roi_L)\cap \GLt(L)$ be the set of matrices that are upper-triangular modulo $p$. In particular, we have $\Gamma\subset\Sigma_0(p)$. Then $\A_k$ has a natural left action of $\Sigma_0(p)$, depending on $k$ (justifying the notation), given by
\[\matr\cdot f(z) = (a+cz)^kf\left(\frac{b+dz}{a+cz}\right).\]
This action preserves the subspace $V_k$ and hence gives rise to component-wise actions of $\Sigma_0(p)^2$ on $V_k\otimes V_k$, $V_k\otimes\A_k$ and $\A_k\otimes \A_k$. Accordingly, we get right actions of $\Sigma_0(p)^2$ on their corresponding topological duals $\VVkdual$, $\VDk$ and $\DDk$ respectively. By dualising the inclusions, we get $\Sigma_0(p)^2$-equivariant surjections
\[\DDk \labelrightarrow{\pr_2} \VDk \labelrightarrow{\pr_1} \VVkdual,\]
that induce maps
\[\h^1(\Gamma,\DDk)\labelrightarrow{\rho_2}\h^1(\Gamma,\VDk)\labelrightarrow{\rho_1}\h^1(\Gamma,\VVkdual)\]
on the cohomology.
$\lb$
We define filtrations as follows:
\begin{mdef}\begin{itemize}
\item[(i)]Let $N$ be an integer and define $\f^N\D_k \defeq \{\mu\in\D_k:\mu(z^r) \in \pi_L^{N-r}\roi_L\}$, where $\pi_L$ is a uniformiser in $\roi_L$. Then define
\[\f^N[\VDk] \defeq V_k^*\otimes \f^N\D_k.\]
This is $\Sigma_0(p)$-stable by arguments in \cite{Gre07} and \cite{Wil17}. Now define
\[F^N[\VDk] \defeq \f^N[\VDk] \cap \ker(\pr_1),\]
which is also $\Sigma_0(p)$-stable as $\pr_1$ is $\Sigma_0(p)$-equivariant.
\item[(ii)]Similarly, define $F^N[\DDk] \defeq (\f^N\DDk)\cap\ker(\pr_2),$ which again is $\Sigma_0(p)$-stable.
\end{itemize}
\end{mdef}
Let $\alpha \in \roi_L$ and let $\pi_{\pri}\defeq[\smallmatrd{1}{0}{0}{1},\smallmatrd{1}{0}{0}{p}]$ and $\pi_{\pribar}\defeq[\smallmatrd{1}{0}{0}{p},\smallmatrd{1}{0}{0}{1}]\in \Sigma_0(p)^2.$ First, we have:
\begin{mlem}Suppose $v_p(\alpha)<k+1$. Then we have
\begin{itemize}
\item[(i)] If $\mu \in F^N[\VDk]$, then $\mu|\pi_{\pri} \in \alpha F^{N+1}[\VDk].$
\item[(ii)] If $\mu \in F^N[\DDk]$, then $\mu|\pi_{\pribar} \in \alpha F^{N+1}[\DDk].$
\end{itemize}
\end{mlem}
We then define the analogue of the module $D^\alpha$ as follows:
\begin{mdef}
\begin{itemize}
\item[(i)]Define $\D_k^\alpha \defeq \{\mu\in\D_k: \mu(z^r) \in \alpha p^{-r}\roi_L\}$, and then define 
\[ [\VDk]^\alpha \defeq V_k^*\otimes \D_k^\alpha.\]
\item[(ii)]Similarly, define $[\DDk]^\alpha \defeq \D_k^\alpha\widehat{\otimes}\D_k$.
\end{itemize}
\end{mdef}
\begin{mlem}\begin{itemize}
\item[(i)] If $\mu\in[\VDk]^{\alpha}$, then $\mu|\pi_{\pri} \in \alpha\VDk.$
\item[(ii)] If $\mu\in[\DDk]^{\alpha}$, then $\mu|\pi_{\pribar}\in\alpha\DDk.$
\end{itemize}
\end{mlem}
Accordingly, we can lift using Theorem \ref{liftingtheorem}, first along $\rho_1$ using the operator $U_{\pri}$ induced by $\pi_{\pri}$, and secondly along $\rho_2$ using the operator $U_{\pribar}$ induced by $\pi_{\pribar}$. In particular, we have:
\begin{mthm}
Let $\alpha_{\pri},\alpha_{\pribar} \in \roi_L$ with $v_p(\alpha_1),v_p(\alpha_2) < k+1.$ Then the restriction of the map $\rho \defeq \rho_2\rho_1$ to the simultaneous $\alpha_{\pri}$ and $\alpha_{\pribar}$ eigenspaces of the $U_{\pri}$ and $U_{\pribar}$ operators respectively is an isomorphism.
\end{mthm}
\begin{mrems}
\begin{itemize}
\item[(i)] In \cite{Wil17}, these results are used to construct $p$-adic $L$-functions for automorphic forms for $\GLt$ over an imaginary quadratic field, in the spirit of \cite{PS11}. In particular, we associate to such an automorphic form a canonical element in the overconvergent cohomology, from which we can very naturally build a ray class distribution that interpolates $L$-values of the automorphic form. It would be interesting to know if similar results existed in the case of $\SLthree$.
\item[(ii)]In the interests of transition to the case of $\SLthree$, we can rephrase the above definitions in a more abstract way. In particular, let $G\defeq \mathrm{Res}_{K/\Q}\GLt$, with Borel subgroup $B$ and opposite Borel $B^{\mathrm{opp}}$. Define $T$ to be the torus, and note we can view $\lambda$ as a dominant weight for $T$, and that $V_k\otimes V_k$ is the representation of $\GLt$ of highest weight $\lambda$ with respect to $B^{\mathrm{opp}}$. Note that for an extension $L/\Qp$, we have  $G(L) \cong \GLt(L)\times\GLt(L)$. Then $\AAk$ is the ring of analytic functions on $B(L)$ that transform like $\lambda$ under multiplication by elements of $T(L)$, whilst $\VAk$ is the ring of analytic functions on $B(L)$ that transform like $\lambda$ under multiplication by elements of $\GLt(L)\times T_{\Q}(L)$, where $T_{\Q}(L)$ is the torus of diagonal matrices in the algebraic group $\GLt/\Q$. In particular, the definitions in the following section are a natural analogue of the theory described concretely above.
\end{itemize}
\end{mrems}

%
%
%
%

\section{Overconvergent modular symbols for $\mathrm{SL}_3$}
We now apply the results above to give a generalisation of the lifting theorem for $\SLthree$ of Pollack and Pollack in \cite{PP09}. We first recall the setting, and also develop the notion of `partially overconvergent' modular symbols for $\SLthree$.

\subsection{Notation}\label{notation}
We recall the setting; where possible, we keep to the notation used by Pollack and Pollack in \cite{PP09} for clarity. For further details, the reader is directed to their paper. Let $G$ be the algebraic group $\mathrm{GL}_3/\Q$, and denote by $B$ (resp.\ $B^{\mathrm{opp}}$) its Borel subgroup of upper-triangular (resp.\ lower-triangular) matrices, with $T$ and $N$ (resp.\ $N^{\mathrm{opp}}$) the subgroups of $B$ (resp.\ $B^{\mathrm{opp}}$) consisting of the diagonal and unipotent matrices respectively. Note that $B = TN$. Let $p$ be a prime, let $\Gamma_0(p)$ be the subgroup of $\SLthree(\Z)$ of matrices that are upper-triangular modulo $p$, and let $\Gamma$ be a congruence subgroup of $\SLthree(\Z)$ contained in $\Gamma_0(p)$.

\subsection{Classical coefficient modules}\label{classicalcoeffs}
Let $\lambda$ be a dominant algebraic character of the torus $T$, which can be seen as an element $\lambda = (k_1,k_2,k_3) \in \Z^3.$ Let $V_\lambda$ be the (unique) representation of $G$ with highest weight $\lambda$ with respect to $B^{\mathrm{opp}}$; for example, when $\lambda = (k,0,0)$, we see that $V_\lambda(A)$ is nothing but $\mathrm{Sym}^k(A^3)$, for a suitable coefficient module $A$. 
\begin{mrem}
We will restrict to the case where $\lambda = (k_1,k_2,0)$, rescaling by the determinant, since this slightly simplifies the calculations. Indeed, any such weight can be written in the form $\lambda = (k_1+v,k_2+v,v)$, and then $V_\lambda \cong V_{\lambda'}\otimes\det^v$, where $\lambda' = (k_1,k_2,0)$. All of our main results then go through in the general case with only slight modification, and indeed, the range of `non-criticality' for the slope for $\lambda'$ is the same as that for $\lambda$ scaled by $v$ in each component.
\end{mrem}

\subsection{Overconvergent coefficient modules}\label{ovcgtcoeffs}
We denote by $\Cp$ the completion of fixed algebraic closure of $\Qp$, and write $\roi_{\Cp}$ for its ring of integers. We now define two different overconvergent coefficient modules corresponding to two different parabolic subgroups of $\SLthree$.

\subsubsection{Overconvergent with respect to $T = \SLo^3$}
We first look at the case where we consider the parabolic subgroup $T = \mathrm{SL}_1\times\mathrm{SL}_1\times\mathrm{SL}_1.$ This identically mirrors the work of Pollack and Pollack in \cite{PP09}. In particular, let $\mathcal{I}$ denote the subgroup of $G(\roi_{\Cp})$ of matrices that are upper-triangular modulo the maximal ideal of $\roi_{\Cp}$. 
$\lb$
We consider continuous function $f: B(\roi_{\Cp}) \rightarrow \roi_{\Cp}$ satisfying the condition
\begin{equation}\label{lambda}
f(tb) = \lambda(t)f(b),\hspace{12pt}t \in T(\roi_{\Cp}), b\in B(\roi_{\Cp})
\end{equation}
We note that any such function is determined by its restriction to $N(\roi_{\Cp})$, and that we can identify $N(\roi_{\Cp})$ with $\roi_{\Cp}^3$ by identifying
\[\begin{pmatrix}1 & x & y\\
0 & 1 & z\\
0 & 0 & 1
\end{pmatrix} \longleftrightarrow (x,y,z) \in \roi_{\Cp}^3.\]
We write $f(x,y,z)$ for the image of this matrix under $f$.
$\lb$
Let $L/\Qp$ be a finite extension with ring of integers $\roi_L$. We say that such a function $f$ is \emph{$L$-rigid analytic} if, for $(x,y,z) \in N(\roi_{\Cp})$, we can write $f$ in the form
\[f(x,y,z) = \sum_{r,s,t\geq 0}c_{rst}x^ry^sz^t,\]
where $c_{rst} \in L$ tends to $0$ as $r+s+t \rightarrow \infty$. Alternatively, this occurs if and only if $f(x,y,z) \in L\langle x,y,z \rangle$, the Tate algebra in three variables over $L$. Writing $\roi_L$ for the ring of integers of $L$, there is likewise an integral version with $c_{rst} \in\roi_L.$ 
\begin{mrem}
Henceforth, we will state all definitions and results in terms of coefficients in $\roi_L$, since in the sequel we use this integrality in an essential way to define filtrations. We could easily instead state the definitions using $L$ in place of $\roi_L$.
\end{mrem}

\begin{mdef}\begin{itemize}\item[(i)]Write $\A_\lambda(\roi_L)$ for the space of $\roi_L$-rigid analytic functions on $B(\roi_{\Cp})$ that satisfy equation (\ref{lambda}).
\item[(ii)] Let $\D_\lambda(\roi_L)$ denote the topological dual
\[\D_\lambda(\roi_L)\defeq \Hom_{\cts}(\A_\lambda(\roi_L),\roi_L)\]
(resp.\ $\Hom_{\cts}(\A_\lambda(\roi_L),\roi_L)$), the space of \emph{rigid analytic distributions on $B(\roi_{\Cp})$ of weight $\lambda$}.
\end{itemize}
\end{mdef}
In an abuse of notation, we write $x^ry^sz^t$ for the unique extension to $B(\roi_{\Cp})$ of the function on $N(\roi_{\Cp})$ that sends
\[\begin{pmatrix}1&x&y\\0&1&z\\0&0&1\end{pmatrix} \longmapsto x^ry^sz^t,\]
and note that any $\mu \in \D_\lambda(\roi_L)$ is uniquely determined by its values at $x^ry^sz^t$ for $r,s,t\geq 0$. Pollack and Pollack call this function $f_{rst}$.

\subsubsection{Overconvergent with respect to $P \defeq \mathrm{SL}_1\times\SLt$}
We now define a different module of overconvergent coefficients. This is, in a sense, a \emph{smaller} module of coefficients, and will play the role of `half-overconvergent' coefficients in the following.
$\lb$
Let $P = \SLo\times\SLt \subset \SLthree$. If $\lambda = (k_1,k_2,0)$ with $k_1\geq k_2$, we get an associated representation 
\[W_\lambda(A) \defeq \mathrm{det}^{k_1}\otimes\hspace{1pt}\mathrm{Sym}^{k_2}(A^2)\]
of $P(A) = \SLo(A)\times\SLt(A)$, for suitable $A$. We can replace $B$ with the larger subgroup $B_1$ of matrices that are block lower-triangular with respect to this parabolic subgroup -- that is, matrices that are zero in the $(2,1)$ and $(3,1)$ entries -- and consider the space of functions $f: B_1(\roi_{\Cp}) \longrightarrow W_\lambda(\roi_{\Cp})$ satisfying the condition 
\[f(tg) = \lambda(t)f(g) \hspace{12pt}\forall t\in P(\roi_{\Cp}),g\in B_1(\roi_{\Cp}), \hspace{20pt}\text{where }\lambda(t)\in\mathrm{GL}(W_\lambda).\]
Note that any such function is entirely determined by its restriction to $B(\roi_{\Cp})$, and indeed by its values on the subgroup
\[\left\{\begin{pmatrix}1 & x & y\\
0 & 1 & 0\\
0 & 0 & 1
\end{pmatrix} \in B_1(\roi_{\Cp})\right\},\]
by a similar argument to before. We say such a function is $\roi_L$-rigid analytic if it is an element of $\roi_L\langle x,y\rangle \otimes_{L}W_\lambda(L)$.
\begin{mdef}
Write $\A_\lambda^P(\roi_L)$ for the space of $\roi_L$-rigid analytic functions on $B_1(\roi_{\Cp})$ that transform like $\lambda$ under elements of $P$.
\end{mdef}
\begin{mprop}
Let $f\in\A_\lambda^P(\roi_L)$. For $g\in B$, let $P_g(X,Y)\defeq f(g) \in W_\lambda(\roi_L),$ where we consider elements of $W_\lambda$ as homogeneous polynomials of degree $k_2$ in two variables over $\roi_L$. Define a function 
\[f':B(\roi_{\Cp}) \longrightarrow \roi_{\Cp}\]
by $f'(g) = P_g(0,1)$. Then $f' \in \A_\lambda(\roi_L)$. Moreover, the association $f \mapsto f'$ gives an isomorphism
\[\A_\lambda^P(\roi_L) \isorightarrow \bigg\{f(x,y,z) = \sum_{r,s,t\geq 0}\alpha_{r,s,t}x^ry^sz^t\in\A_\lambda(\roi_L): \alpha_{r,s,t} = 0 \text{ for }t > k_2\bigg\}.
\]
\end{mprop}
\begin{proof}
Firstly, note that $f'$ is rigid analytic in three variables. In particular, let
\[g \defeq \begin{pmatrix}1 & x & y\\
0 & 1 & 0\\
0 & 0 & 1
\end{pmatrix} \hspace{12pt}\text{and}\hspace{12pt}
P_g(X,Y) = \sum_{t=1}^{k_2}\sum_{r,s\geq 0}\alpha_{r,s,t}x^ry^sX^tY^{k_2-t},\]
using rigidity of $f$. Then consider
\[g'\defeq\begin{pmatrix}1 & x & y\\
0 & 1 & z\\
0 & 0 & 1
\end{pmatrix} = \begin{pmatrix}1 & 0 & 0\\
0 & 1 & z\\
0 & 0 & 1
\end{pmatrix}\begin{pmatrix}1 & x & y\\
0 & 1 & 0\\
0 & 0 & 1
\end{pmatrix}.\]
Recall that $\GLt(L)$ acts on $W_\lambda(L)$ by $w|\smallmatrd{a}{b}{c}{d}(X,Y) = w(bY+dX,aY+cX),$ so that 
\begin{equation}\label{image}
f'(x,y,z) = P_{g'}(0,1) = P_g(X+z,Y)\bigg|_{X=0,Y=1} = P_g(z,1) = \sum_{t=1}^{k_2}\sum_{r,s\geq 0}\alpha_{r,s,t}x^ry^sz^t.
\end{equation}
The rigidity follows. Now we show that $f'$ transforms under $T$ as $\lambda$. Let $g\in B(\roi_{\Cp})$ and $t = (t_1,t_2,t_3) \in T(\roi_{\Cp})$. Then compute
\[P_{tg}(X,Y) = f(tg)(X,Y) = t_1^{k_1}f(g)(t_3X, t_2Y) = t_1^{k_1}P_g(t_3X,t_2Y).\]
Accordingly, we have 
\begin{align*}f'(tg) = P_{tg}(0,1) = t_1^{k_1}P_g(0,t_2) = t_1^{k_1}t_2^{k_2}P_g(0,1)= \lambda(t)f'(g),
\end{align*}
as required.
$\lb$
Finally, it remains to show that the map induces the stated isomorphism. From equation (\ref{image}), it is clear that $f' = 0$ if and only if $f = 0$, so that the association $f\mapsto f'$ is injective. It is also clear that the image is the right-hand side of the isomorphism. This completes the proof.
\end{proof}
\begin{mdef}\begin{itemize}
\item[(i)]From now on, in an abuse of notation using this isomorphism, we write $\A_\lambda^P(\roi_L)$ for this subspace of $\A_\lambda(\roi_L)$.
\item[(ii)] Let $\D^P_\lambda(\roi_L)$ denote the topological dual
\[\D^P_\lambda(\roi_L)\defeq \Hom_{\cts}(\A^P_\lambda(\roi_L),\roi_L),\]
the space of \emph{rigid analytic distributions on $B(\roi_{\Cp})$ of weight $\lambda$ over $\roi_L$}.
\end{itemize}
\end{mdef}
Note that by dualising the inclusion $\A_\lambda^P(\roi_L)\subset \A_\lambda(\roi_L)$, we get a surjective map 
\[\pr_\lambda^2: \D_\lambda(\roi_L) \longrightarrow \D_\lambda^P(\roi_L),\]
where the notation will become clear in the sequel.
\begin{mremnum}\label{pi}
Note that $\D_\lambda^P(\roi_L)$ is, in a sense, `partially' overconvergent, in the sense that it is overconvergent in the variables $x,y$ and classical in $z$. In the next section, we will introduce operators 
\[\pi_1 \defeq \begin{pmatrix}1 & 0 & 0\\
0 & p & 0\\
0 & 0 & p
\end{pmatrix}\hspace{12pt}\text{and}\hspace{12pt}\pi_2 \defeq \begin{pmatrix}1 & 0 & 0\\
0 & 1 & 0\\
0 & 0 & p
\end{pmatrix},\]
whose product is the element $\pi$ considered by Pollack and Pollack in \cite{PP09}. We will ultimately lift a classical modular symbol to one that takes values in $\D_\lambda^P(\roi_L)$ using $\pi_1$, and then lift this further to a symbol that takes values in the space $\D_\lambda(\roi_L)$ of fully overconvergent coefficients using $\pi_2$.
\end{mremnum}

\subsection{The action of $\Sigma$ and specialisation}
\subsubsection{The weight $\lambda$ action}
Let $\mathscr{X}$ denote the image of the Iwahori group $\mathcal{I}$ in $N^{\mathrm{opp}}(\roi_{\Cp})\backslash G(\roi_{\Cp})$ under the natural embedding, and note that we can identify $\mathscr{X}$ with $B(\roi_{\Cp})$ in a natural way. Let 
\[I \defeq \mathcal{I}\cap \SLthree(\Z). \]
(Note that $I=\Gamma_0(p)$ in this setting, though we retain the notation for ease of comparison with Pollack and Pollack.) We also define $\pi_1$ and $\pi_2$ as in Remark \ref{pi}, and let $\Sigma$ be the semigroup generated by $I$, $\pi_1$ and $\pi_2$.
$\lb$
Note that $I$ acts on $N^{\mathrm{opp}}(\roi_{\Cp})\backslash G(\roi_{\Cp})$ by right multiplication, and as $\pi$ normalises $N^{\mathrm{opp}}$, we also have a right action of $\pi$ on this space by
\[N^{\mathrm{opp}}(\roi_{\Cp})g|\pi = N^{\mathrm{opp}}(\roi_{\Cp})\pi^{-1}g\pi.\]
Thus we have an action of $\Sigma$ on this space. This action preserves $\mathscr{X}$ and hence gives rise to a right action of $\Sigma$ on $B(\roi_{\Cp})$. This in turn gives a left action of $\Sigma$ on $\A_\lambda(\roi_L)$ by $\gamma\cdot f(b) = f(b|\gamma),$ and dually a right action of $\Sigma$ on $\D_\lambda(\roi_L)$ by $\mu|\gamma(f) = \mu(\gamma\cdot f).$ In \cite{PP09}, Lemma 2.1, Pollack and Pollack give an explicit description of this action. We recap their results: 
\begin{mlem}\label{explicitaction}\begin{itemize}
\item[(i)]
Let $\lambda = (k_1,k_2,0)$. For $\gamma \in I$, the weight $\lambda$ action of $\gamma$ on $f\in\A_\lambda(\roi_L)$ is given by
\begin{align*}(\gamma f))(x,y,z) = &(a_{11}+a_{21}x + a_{31}y)^{k_1-k_2}(m_{33}-m_{13}y-m_{23}z +m_{13}xz)^{k_2} \times\\
&f\bigg(\frac{a_{12}+a_{22}x+a_{32}y}{a_{11}+a_{21}x + a_{31}y},\frac{a_{13}+a_{23}x+a_{33}y}{a_{11}+a_{21}x + a_{31}y}, \\
&\hspace{100pt}\frac{-m_{32} + m_{12}y + m_{22}z - m_{12}xz}{m_{33}-m_{13}y-m_{23}z+m_{13}xz}\bigg),
\end{align*}
where $\gamma = (a_{ij})$ and $m_{ij}$ is the $(i,j)$th minor of $\gamma$.
\item[(ii)]We have 
\[\pi_1\cdot f(x,y,z) = f(px,py,z)\]
and
\[\pi_2\cdot f(x,y,z) = f(x,py,pz).\]
\end{itemize}
\end{mlem}
\begin{proof}
For part (i), see \cite{PP09}, Lemma 2.1. For part (ii), this is easily checked by computing
\[\pi_1^{-1}\begin{pmatrix}1&x&y\\0&1&z\\0&0&1\end{pmatrix}\pi_1 = \begin{pmatrix}1&px&py\\0&1&z\\0&0&1\end{pmatrix}.\]
The case of $\pi_2$ is done similarly.
\end{proof}
\begin{mprop}
The action of $\Sigma$ preserves the subspace $\A_\lambda^P(\roi_L)$ of $\A_\lambda(\roi_L)$.
\end{mprop}
\begin{proof}
The space $\A_\lambda^P(\roi_L)$ is the span of the functions $x^ry^sz^t$ with $t\leq k_2$ (under suitable restrictions on the coefficients). So it suffices to show that $\gamma \cdot x^ry^sz^t$ lies in this span. But from Lemma \ref{explicitaction} above, this is clear.
\end{proof}
\begin{mcor}
The map $\D_\lambda(\roi_L) \rightarrow \D_\lambda^P(\roi_L)$ given by dualising the inclusion is equivariant with respect to the action of $\Sigma$.
\end{mcor}

\subsubsection{Specialisation to weight $\lambda$}\label{specialisation}
We want to exhibit a map from overconvergent to classical coefficients, which we'll call \emph{specialisation to weight $\lambda$}. To this end, let $v_\lambda$ be a highest weight vector in $V_\lambda(\roi_L)$ (which we take to be a \emph{right} representation of $G$). More precisely, this is an element satisfying
\[v_\lambda|t = \lambda(t)v_\lambda \hspace{4pt}\forall t \in T(\roi_L), \hspace{12pt}v_\lambda|n = v_\lambda \hspace{4pt}\forall n \in N^{\mathrm{opp}}(\roi_L).\]
In particular, we can define a map 
\begin{align*}
f_\lambda: G(\roi_L) &\longrightarrow V_\lambda(\roi_L)\\
g &\longmapsto v_\lambda|g.
\end{align*}
Since we have invariance under $N^{\mathrm{opp}}$, this function descends to $N^{\mathrm{opp}}\backslash G$. We can then restrict this function to (the $\roi_L$-points of) $\mathscr{X}$.
\begin{mlem}
Let $\lambda = (k_1,k_2,0)\in\Z^3.$ Then $V_\lambda(\roi_L)$ can be realised as a subrepresentation of $\mathrm{Sym}^{k_1}(\roi_L^3)\otimes \mathrm{Sym}^{k_2}(\roi_L^3)$, and the highest weight vector is
\[v_\lambda = \sum_{i=0}^{k_2}(-1)^i\binomc{k_1}{i}X^{k_1-i}Y^i\otimes U^iV^{k_2-i},\]
where a general element has form $\sum P(X,Y,Z)\otimes Q(U,V,W).$
\end{mlem}
\begin{proof}
See \cite{PP09}, Remark 2.4.3.
\end{proof}
\begin{mprop}
We have $f_\lambda\bigg|_{\mathscr{X}} \in \A^P_\lambda(\roi_L)\otimes V_\lambda(\roi_L)$.
\end{mprop}
\begin{proof}We explicitly compute $v_\lambda|g$, where 
\[g = \threematrix{1}{x}{y}{1}{z}{1}.\]
We see that this is equal to 
\begin{equation}\label{rholambda1}
v_\lambda|g = \sum_{i=1}^{k_2}(-1)^i\binomc{k_2}{i}(X+xY+yZ)^{k_1-i}(Y+zZ)^i\otimes (U+xV+yW)^i(V+zW)^{k_2-i}.
\end{equation}
It's easy to see from this that the coefficient of each monomial is an element of $\A_\lambda^P(\roi_L)$ (and in particular that the maximal degree of $z$ in this expression is $k_2$), and we conclude the result.
\end{proof}
For a distribution $\mu\in\D^P_\lambda(\roi_L)$, define an `evaluation at $\A^P_\lambda(\roi_L)\otimes V_\lambda(\roi_L)$' map by setting
\[\mu(f\otimes v) = \mu(f)\otimes v \in V_\lambda(\roi_L).\]
In particular, we can evaluate at $f_\lambda$.
\begin{mdef}Define the \emph{specialisation map at weight $\lambda$} to be the map 
\[\pr_\lambda^1 : \D^P_\lambda(\roi_L) \longrightarrow V_\lambda(\roi_L)\]
given by evaluation at $f_\lambda \in \A^P_\lambda(\roi_L)\otimes V_\lambda(\roi_L).$
\end{mdef}
This map is $I$-equivariant, but not $\pi_i$-equivariant. As in \cite{PP09}, we introduce a twisted action of $\pi_i$ to get around this.
\begin{mdef} \label{twisted}
Define a (right) action of $\Sigma$ on $V_\lambda(L)$ by 
\begin{align*}
v\star \gamma &= v|\gamma, \hspace{4pt}\gamma \in I,\\
v\star \pi_i &= \lambda(\pi_i)^{-1}v|\pi_i.
\end{align*}
Let $V_\lambda^\star(L)$ denote the module $V_\lambda(L)$ with this twisted action.
\end{mdef}
 Then we see that:
\begin{mlem}
The map $\pr_\lambda^1: \D^P_\lambda(L) \longrightarrow V_\lambda^\star(L)$ is $\Sigma$-equivariant.
\end{mlem}
\begin{mdef}
Let $L_\lambda(\roi_L) \defeq \pr_\lambda^1(\D_\lambda^P(\roi_L))\subset V_\lambda^\star(L).$ Note that this is stable under the $\star$-action of $\Sigma$ since $\pr_\lambda^1$ is $\Sigma$-equivariant.
\end{mdef}
We have an action of $\Gamma \subset I$ on these coefficient spaces. In particular, we can define the group cohomology of these coefficient spaces, and then note that, for each integer $r$, the map $\pr_\lambda^1$ induces a map
\[\rho_\lambda^1 \defeq \rho_\lambda^1(r): \h^r(\Gamma,\D^P_\lambda(\roi_L)) \longrightarrow \h^r(\Gamma,L_\lambda(\roi_L)).\]
These spaces come equipped with the natural Hecke action on cohomology, and the action of the $U_p$ operator is given by the matrix $\pi = \pi_1\pi_2$.

%
%
%
%

\section{Filtrations and control theorems for $\SLthree$}
We recall what we have done so far. For a weight $\lambda = (k_1,k_2,0) \in \Z^3$, we defined a space $L_\lambda(\roi_L)$ of classical coefficients, a space $\D_\lambda^P(\roi_L)$ of partially overconvergent coefficients, and a space $\D_\lambda(\roi_L)$ of fully overconvergent coefficients (where $\D_\lambda(\roi_L)$ is as defined in \cite{PP09}). We also defined maps $\pr_\lambda^i$ between these coefficient modules, and these induce maps
\[ \h^r(\Gamma,\D_\lambda(\roi_L)) \labelrightarrow{\rho_\lambda^2} \h^r(\Gamma, \D_\lambda^P(\roi_L)) \labelrightarrow{\rho_\lambda^1} \h^r(\Gamma,L_\lambda(\roi_L))\]
on the cohomology.
$\lb$
In this section, we prove that if we restrict to the simultaneous \emph{small-slope} eigenspaces of the operators on the cohomology given by $\pi_1$ and $\pi_2$, the composition $\rho_\lambda$ of these maps is an isomorphism. For posterity, we give the definition of small slope now.
\begin{mdef}\label{heckeatp}
Let $U_{p,i}$ be the operator on the cohomology induced by the element $\pi_i$ of Remark \ref{pi}, for $i=1,2$. We call these operators the \emph{Hecke operators at $p$}.
\end{mdef}
\begin{mdef}Let $\phi$ be an eigensymbol at $p$ (with classical or overconvergent coefficients) of weight $\lambda = (k_1,k_2,0)$, and write $U_{p,i}\phi = \alpha_i \phi$ for $i = 1,2$. We say  said to be \emph{small slope} at $p$ if
\[v_p(\alpha_1) < k_1-k_2+1 \hspace{12pt}\text{and} \hspace{12pt} v_p(\alpha_2) < k_2+1.\]
\end{mdef}
In particular, we will show that the restriction of $\rho_\lambda$ to the small slope subspaces is an isomorphism. We use two applications of Theorem \ref{liftingtheorem} to prove this.

\subsection{Lifting to partially overconvergent coefficients}
We now define a filtration on the modules $\D_\lambda^P(\roi_L)$ that allows us to apply Theorem \ref{liftingtheorem}.
\subsubsection{Filtrations on $\D_\lambda^P(\roi_L)$}
\begin{mdef}Define 
\[F^N\D_\lambda^P(\roi_L) \defeq \left\{\mu \in \D_\lambda^P(\roi_L): \mu(x^ry^sz^t) \in \pi_L^{N-(r+s)}\roi_L\right\}.\]
\end{mdef}
\begin{mprop}\label{sigmastable}
The filtration $F^N\D_\lambda^P(\roi_L)$ is stable under the action of $\Sigma$. 
\end{mprop}
\begin{proof}
Let $\mu\in F^N\disp.$ We know that, for $\gamma = (a_{ij}) \in I$, we have
\begin{align*}\gamma\cdot x^ry^sz^t = &(a_{12}+a_{22}x+a_{32}y)^r(a_{13}+a_{23}x+a_{33}y)^s\\ 
&\times(-m_{32}+m_{22}z-(m_{12}z)x+m_{12}y)^t(a_{11}+a_{21}x+a_{31}y)^{k_1-k_2-r-s}\\
&\hspace{60pt}\times(m_{33}-m_{23}z-(m_{13}z)x-m_{13}y)^{k_2-t},
\end{align*}
where $m_{ij}$ is the $(i,j)$th minor of $\gamma$, using Lemma \ref{explicitaction}. Write this as 
\[\mu|\gamma(x^ry^sz^t) = \sum_{a,b\geq 0}\beta_{ab}(z)x^ay^b,\]
where $\beta_{ab}(z)$ is a polynomial in $z$ of degree at most $t$. Then note that $p$ divides the terms $a_{21},a_{31},a_{32},m_{12},m_{13},$ and $m_{23}$, whilst the terms $a_{11},a_{22},a_{33},m_{22}$ and $m_{33}$ are all $p$-adic units. In particular, we examine the $p$-divisibility conditions on the coefficients $\beta_{ab}(z)$. Any monomial $x^ay^b$ coming from the first bracket in this expression has coefficient divisible by $p^{a+b-r}$, since $p|a_{32}$. Similarly, any such monomial in the second bracket has coefficient divisible by $p^{a+b-s}$. Moreover, since in the remaining three brackets $p$ divides the coefficient of both $x$ and $y$ before expanding, we see that any monomial including $x^ay^b$ in the expanded expression is divisible by $p^{a+b}$. Accordingly, by combining this, we see that $p^{a+b-(r+s)}|\beta_{ab}(z)$. Since we already know that $\mu(x^ay^bz^c) \in \pi_L^{N-(a+b)}\roi_L$ for any $c\leq t$, we now see that
\[\mu(\beta_{ab}(z)x^ay^b) \in p^{a+b-(r+s)}\pi_L^{N-(a+b)}\roi_L \subset \pi_L^{N-(r+s)}\roi_L,\]
as required.
$\lb$
Since $\pi_1$ and $\pi_2$ act on such monomials by multiplying by a non-negative power of $p$, they also preserve the filtration. Thus the filtration is stable under the action of $\Sigma$.
\end{proof}
We actually need a slightly finer filtration.
\begin{mdef}
Define
\[\f^N\D_\lambda^P(\roi_L) \defeq F^N\D_\lambda^P(\roi_L)\cap \ker(\pr_\lambda^1).\]
\end{mdef}
Since $\pr_\lambda^1$ is $\Sigma$-equivariant, this filtration is also $\Sigma$-stable. The crux of our argument is then:
\begin{mprop}\label{kernelcond}
Suppose $\mu\in \ker(\pr_\lambda^1).$ Then 
\[\mu(x^ry^sz^t) = 0 \hspace{12pt}\text{for all }r+s \leq k_1-k_2, \hspace{4pt}0\leq t\leq k_2.\]
\end{mprop}
\begin{proof}
We explicitly examine the map $\pr_\lambda^1$. Earlier, in equation (\ref{rholambda1}), we gave a formula for the expression $f_\lambda(x,y,z)$. If $\mu \in \ker(\pr_\lambda^1)$, then in particular $\mu(f_\lambda(x,y,z)) = 0$. We consider the monomials including the term $U^{k_2}$, keeping the notation of previously. Such a term can occur only for $i=k_2$, so that these terms all appear in 
\[(-1)^{k_2}(X+xY+zZ)^{k_2-k_1}(Y+zZ)^{k_2}\otimes U^{k_2}.\]
By expanding out this bracket, and considering the coefficients of each monomial, we see that we have $\mu(x^ry^sz^t) = 0$ for at least the range of $r,s$ and $t$ specified by the proposition. 
\end{proof}
\begin{mrem}Note that, for general $\lambda$, this condition on $r+s$ is optimal. In particular, consider $\lambda = (k,1,0)$, for some integer $k\geq 1$. Then if $\mu\in\ker(\rho_\lambda^1)$, then we do not necessarily have $\mu(x^k) = 0$, so in particular we can't say anything general about the values $\mu(x^ry^s)$ where $r+s>k-1$.
\end{mrem}
This filtration satisfies the conditions of Theorem \ref{liftingtheorem}, as we see by:
\begin{mlem}\label{partb}
Let $\mu \in \fil$, and let $\alpha \in \roi_L$ with $v_p(\alpha)<k_1-k_2+1$. Then \[\mu|\pi_1 \in \alpha\f^{N+1}\D_\lambda^P(\roi_L).\]
\end{mlem}
\begin{proof}
We have $\mu|\pi_1(x^ry^sz^t) = p^{r+s}\mu(x^ry^sz^t)$. From Proposition \ref{kernelcond}, we see that if $r+s\leq k_1-k_2$, we have $\mu(x^ry^sz^t) = 0$. In particular, from this, we have
\[\mu|\pi_1(x^ry^sz^t) \in p^{k_1-k_2+1}\pi_L^{N-(r+s)}\roi_L.\]
As $v_p(\alpha)<k_1-k_2+1$, and it must be divisible by an integral power of $\pi_L$, we have $p^{k_1-k_2+1}\in \alpha \pi_L\roi_L$, so that
\[\mu|\pi_1(x^ry^sz^t) \in \alpha\pi_L^{1+N-(r+s)}\roi_L.\]
Thus $\mu \in \alpha\f^{N+1}\D_\lambda^P(\roi_L)$, as required.
\end{proof}

\subsubsection{A submodule of $\D_\lambda^P(\roi_L)$}
We require one further definition before we can apply Theorem \ref{liftingtheorem}; namely, a submodule of $\D_\lambda^P(\roi_L)$ that will play the role of $D^\alpha$ in condition (v) in Notation \ref{setup}.
\begin{mdef}
Let $\alpha\in\roi_L$. Define 
\[\D_\lambda^{P,\alpha}(\roi_L) \defeq \left\{\mu\in\D_\lambda^P(\roi_L): \mu(x^ry^sz^t)\in \alpha p^{-(r+s)}\roi_L\right\}.\]
\end{mdef}
\begin{mprop}
The subspace $\D_\lambda^{P,\alpha}(\roi_L)$ is stable under the action of $\Sigma$.
\end{mprop}
\begin{proof}
Let $\mu \in \D_\lambda^P(\roi_L)$ and $\gamma\in I$, and recall the proof of Proposition \ref{sigmastable}, and in particular, the computation
\[\mu|\gamma(x^ry^sz^t) = \sum_{a,b\geq 0}\mu(\beta_{ab}(z)x^ay^b),\]
where $p^{a+b-(r+s)}|\beta_{ab}(z)$. Now take $\mu$ to be in the smaller space $\D_\lambda^{P,\alpha}(\roi_L)$. Then $\mu(\beta_{ab}(z)x^ay^b \in p^{a+b-(r+s)}\alpha p^{-(a+b)}\roi_L = \alpha p^{-(r+s)}.$ Thus $\mu|\gamma \in \D_\lambda^{P,\alpha}(\roi_L)$, as required.
$\lb$
As $\pi_1$ and $\pi_2$ act on monomials by multiplying by non-negative powers of $p$, stability in this case is clear.
\end{proof}
\begin{mlem}\label{parta}
Suppose $\mu\in\D_\lambda^{P,\alpha}(\roi_L)$. Then 
\[\mu|\pi_1 \in \alpha\D_\lambda^P(\roi_L).\]
\end{mlem}
\begin{proof}
Consider $\mu|\pi_1(x^ry^sz^t) = p^{r+s}\mu(x^ry^sz^t)$. Since $\mu \in \D_\lambda^{P,\alpha}(\roi_L)$, we see that $\mu|\pi_1(x^ry^sz^t) \in \alpha\roi_L,$ and the result immediately follows.
\end{proof}

\subsubsection{Summary and results}
We can now apply Theorem \ref{liftingtheorem} to the small slope subspace in this situation. In particular, in the set-up of this theorem, let $D = \D_\lambda^P(\roi_L)$ and $D^\alpha = \D_\lambda^{P,\alpha}(\roi_L)$. Then we have written down a filtration of this space that satisfies the conditions of Theorem \ref{liftingtheorem}. In particular, we have all the objects of Notation \ref{setup} (i)-(v), and then we've shown condition (a) of the theorem in Lemma \ref{parta} and condition (b) in Lemma \ref{partb}. So we've proved:
\begin{mprop}\label{partial}
Let $\alpha\in\roi_L$ with $v_p(\alpha)<k_1-k_2+1.$ Let $\D_\lambda^P(\roi_L)$ be the module of partially overconvergent coefficients defined in Section \ref{ovcgtcoeffs}, and let $L_\lambda(\roi_L) = \pr_\lambda^1(\D_\lambda^P(\roi_L))$. Then the restriction 
\[\rho_\lambda^1 : \h^r(\Gamma, \D_\lambda^P(\roi_L))^{U_{p,1}=\alpha} \isorightarrow \h^r(\Gamma,L_\lambda(\roi_L))^{U_{p,1}=\alpha}\]of $\rho_\lambda^1$ to the $\alpha$-eigenspaces of the $U_{p,1}$ operator is an isomorphism.
\end{mprop}

\subsection{From partial to fully overconvergent coefficients}
We now change direction and focus on the action of the $U_{p,2}$ operator induced from $\pi_2$. In particular, by applying the theorem again with the $U_{p,2}$ operator, we can lift from partial to fully overconvergent coefficients. As the results are very similar to, and in many cases simpler than, those above, we present the material here in less detail. 
$\lb$
Define a filtration on $\D_\lambda(\roi_L)$ by
\[\f^N\D_\lambda(\roi_L) \defeq \left\{\mu \in \D_\lambda(\roi_L): \mu(x^ry^sz^t) \in \pi_L^{N-t}\roi_L\right\}\cap\ker(\pr_\lambda^2).\]
This is $\Sigma$-stable by a very similar argument to previously. We also define, for $\alpha\in\roi_L$,
\[\D_\lambda^\alpha(\roi_L) \defeq \left\{\mu\in\D_\lambda(\roi_L): \mu(x^ry^sz^t)\in \alpha p^{-t}\roi_L\right\},\]
which is also easily seen to be $\Sigma$-stable and satisfies the conditions required of $D^\alpha$ in Theorem \ref{liftingtheorem}. When $v_p(\alpha)<k_2+1,$ we see that if $\mu\in\f^N\D_\lambda(\roi_L)$, then $\mu|\pi_2 \in \alpha\f^{N+1}\D_\lambda(\roi_L)$, again by a similar argument before after studying the kernel of $\pr_\lambda^2$. Putting this together and using Theorem \ref{liftingtheorem}, we get:

\begin{mprop}\label{full}
Let $\alpha \in \roi_L$ with $v_p(\alpha)<k_2+1$. Let $\D_\lambda(\roi_L)$ and $\D_\lambda^P(\roi_L)$ be the modules of fully and partially overconvergent coefficients respectively, as defined in Section \ref{ovcgtcoeffs}. Then the restriction 
\[\rho_\lambda^2 : \h^r(\Gamma, \D_\lambda(\roi_L))^{U_{p,2}=\alpha} \isorightarrow \h^r(\Gamma,\D_\lambda^P(\roi_L))^{U_{p,2}=\alpha}\]
of $\rho_\lambda^2$ to the $\alpha$-eigenspaces of the $U_{p,2}$ operator is an isomorphism.
\end{mprop}


\subsection{Summary of results}
We can combine the results of Propositions \ref{partial} and \ref{full} to obtain the following constructive non-critical slope control theorem for $\SLthree$.

\begin{mthm}\label{controltheorem}Consider the set-up of Notation \ref{fullnot} in the Introduction. In particular, let $\lambda = (k_1,k_2,0)$ be a dominant algebraic weight, and let $\alpha_1,\alpha_2 \in \roi_L$ with $v_p(\alpha_1)<k_1-k_2+1$ and $v_p(\alpha_2) < k_2+1.$ Then the restriction
\[\rho_\lambda :\h^r(\Gamma,\D_\lambda(\roi_L))^{U_{p,i}=\alpha_i} \rightarrow \h^r(\Gamma,L_\lambda(\roi_L))^{U_{p,i}=\alpha_i}\]
of the specialisation map to the simultaneous $\alpha_i$-eigenspaces of the $U_{p,i}$ operators, for $i = 1,2$, is an isomorphism.
\end{mthm}
\begin{proof}
This is an immediate consequence of the two propositions. Indeed, both $\rho_\lambda^1$ and $\rho_\lambda^2$ are $\Sigma$-equivariant, so that a partial lift of a simultaneous $U_{p,1}$ and $U_{p,2}$ eigensymbol will likewise be a simultaneous eigensymbol, that can hence be lifted further to fully overconvergent coefficients.
\end{proof}

\small
\renewcommand{\refname}{\normalsize References} 
\bibliography{references}{}
\bibliographystyle{alpha}
\end{document}